\documentclass{amsart}
\usepackage{graphicx} 
\usepackage{amsmath,amssymb,amsthm, mathrsfs}
\usepackage{boxedminipage}
\usepackage[margin=1in]{geometry}
\usepackage{mathpazo}
\usepackage{fancyhdr}
\usepackage[utf8]{inputenc}
\usepackage{color}
\usepackage{tikz}% Required for inserting images
\usetikzlibrary{positioning}
\usepackage{}

\newtheorem{theorem}{Theorem}[section]
\newtheorem{proposition}[theorem]{Proposition}
\newtheorem{lemma}[theorem]{Lemma}
\newtheorem{definition}[theorem]{Definition}
\newtheorem{conjecture}[theorem]{Conjecture}

\newcommand{\gon}{\mathrm{gon}}
\newcommand{\rk}{\mathrm{rk}}

\title{On the Gonality of Ferrers Rook Graphs}
\author{David Jensen}
\author{Marissa Morvai}
\author{Noah Speeter}
\author{William Welch}
\author{Sydney Yeomans}
\begin{document} 
\bibliographystyle{alpha}

\begin{abstract}
A Ferrers rook graph is a graph whose vertices correspond to the dots in a Ferrers diagram, and where two vertices are adjacent if they are in the same row or the same column.  We propose a conjectural formula for the gonality of Ferrers rook graphs, and prove this conjecture for a few infinite families of Ferrers diagrams.  We also prove the conjecture for all Ferrers diagrams $F$ with $\vert F \vert \leq 8$.
\end{abstract}

\maketitle

\section{Introduction}

In this paper, we initiate the study of divisors on Ferrers rook graphs.  A Ferrers rook graph is a graph whose vertices correspond to the elements of a Ferrers diagram, and where two vertices are adjacent if they are in the same row or the same column.  They are a generalization of classical rook graphs, in which the Ferrers diagram is a rectangle.

Our principal question is to compute the \emph{gonality} of these graphs.  The gonality of a graph is a relatively recently defined graph invariant, defined in terms of chip firing games on graphs and with motivation coming from algebraic geometry \cite{Baker08, BakerNorine09}.  In \cite{Speeter}, Speeter computes the gonality of classical rook graphs, which have applications to the study of complete intersection curves.  In \cite{MorrisonSpeeter}, Morrison and Speeter find the gonality of queens graphs, and in \cite{CDDKLMS}, the authors explore the gonality of graphs related to other chess pieces.

In Section~\ref{Sec:Bounds}, we propose a conjectural formula for the gonality of Ferrers rook graphs.  Intuitively, we expect that the divisors of minimal degree and positive rank are given by summing all the vertices in the complement of a non-intersecting row and column, though the edge cases of the \emph{first} row and column require special consideration.  See Conjecture~\ref{Conj:Gonality} for a more precise statement.  While we have not been able to prove this conjecture in its full generality, we have computed several cases.  In Proposition~\ref{Prop:RectanglePlus}, we demonstrate Conjecture~\ref{Conj:Gonality} for Ferrers diagrams that look like a rectangle with one extra partial row, and in Lemma~\ref{Lem:L}, we prove the conjecture for Ferrers diagrams in which every vertex is either in the first row or first column.  As a consequence, we obtain Conjecture~\ref{Conj:Gonality} for all Ferrers diagrams with at most two rows or at most two columns.  In Theorems~\ref{Thm:T3}, \ref{Thm:T4} and~\ref{Thm:T5}, we prove Conjecture~\ref{Conj:Gonality} for isosceles right triangles of sidelength 3, 4, and 5, respectively.
Finally, in Theorem~\ref{Thm:LessThan8}, we prove that Conjecture~\ref{Conj:Gonality} holds for all Ferrers diagrams $F$ with $\vert F \vert \leq 8$.

\section*{Acknowledgements} This research was conducted as a project with the University of Kentucky Math Lab, supported by NSF DMS-2054135.

\section{Preliminaries}

Before we begin, we must establish basic definitions. 

\subsection{Ferrers Rook Graphs}
A \emph{Ferrers diagram} is a finite subset $F \subset \mathbb{N}^2$ with the property that, if $(x,y) \in F$, then either $x=1$ or $(x-1,y) \in F$, and either $y=1$ or $(x,y-1) \in F$.  We will draw Ferrers diagrams in the English style, so that the box $(1,1)$ is in the \emph{top} left corner, and the coordinate $y$ increases from top to bottom.  We write $\vert F \vert$ for the number of elements in $F$.

Given a Ferrers diagram $F$, we define the \emph{Ferrers rook graph} $R(F)$ to be the simple graph whose vertices are elements of $F$, and where two vertices are adjacent if they are in the same row or the same column.  In other words, $(x,y)$ is adjacent to $(x',y')$ if either $x=x'$ or $y=y'$.    We refer to an edge connecting $(x,y)$ to $(x,y')$ as a \emph{vertical edge} and an edge connecting $(x,y)$ to $(x',y')$ as a \emph{horizontal edge}.  For example, Figure~\ref{Fig:Ferrers} depicts a Ferrers diagram $F$ and the corresponding Ferrers rook graph $R(F)$.  Ferrers rook graphs typically have lots of edges, so we will typically draw only the Ferrers diagram $F$, as in the left side of Figure~\ref{Fig:Ferrers}, rather than the graph $R(F)$, as in the right side of Figure~\ref{Fig:Ferrers}.

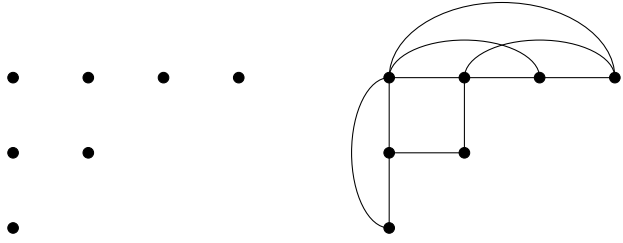
\begin{figure}[ht]
    \vspace{0.25cm}
    \centering
    \begin{tikzpicture}
    \filldraw (0,3) circle (2pt);
    \filldraw (1,3) circle (2pt);
    \filldraw (2,3) circle (2pt);
    \filldraw (3,3) circle (2pt);
    \filldraw (0,2) circle (2pt);
    \filldraw (1,2) circle (2pt);
    \filldraw (0,1) circle (2pt);

    \filldraw (5,3) circle (2pt);
    \filldraw (6,3) circle (2pt);
    \filldraw (7,3) circle (2pt);
    \filldraw (8,3) circle (2pt);
    \filldraw (5,2) circle (2pt);
    \filldraw (6,2) circle (2pt);
    \filldraw (5,1) circle (2pt);
    \draw (5,3)--(8,3);
    \draw (5,3)--(5,1);
    \draw (5,2)--(6,2);
    \draw (6,3)--(6,2);
    \draw (8,3) arc [ start angle=0, end angle=180, x radius=1.5, y radius =1] ;
    \draw (8,3) arc [ start angle=0, end angle=180, x radius=1, y radius =0.5] ;
    \draw (7,3) arc [ start angle=0, end angle=180, x radius=1, y radius =0.5] ;
    \draw (5,3) arc [ start angle=90, end angle=270, x radius=0.5, y radius =1] ;
    
    \end{tikzpicture}
    \caption{A Ferrers diagram $F$ and the corresponding Ferrers rook graph $R(F)$.}
    \label{Fig:Ferrers}
\end{figure}

\subsection{Chip Firing}

In this subsection, we introduce the basic theory of divisors on graphs.  A \emph{divisor} on a graph is an association of an integer to each vertex.  Divisors on a graph $G$ are written as formal sums:
\[
D = \sum_{v \in V(G)} D(v) \cdot v,
\]
where $D(v)$ is the integer associated to $v$.  Divisors can be thought of as stacks of poker chips on the vertices, with negative numbers thought of as a debt of chips.  The \emph{support} of a divisor $D$ is 
\[
\mathrm{Supp}(D) = \{ v \in V(G) \mid D(v) > 0 \}.
\]

In the chip firing game, there is only one kind of move.  One can \emph{fire} a vertex, resulting in that vertex giving one chip to each of its neighbors.  This operation is commutative -- if we fire two vertices, the order in which we fire them does not matter.  For this reason, it is common to talk about firing a set of vertices, meaning that we fire each vertex in the set once, in any order. 

Two divisors are \emph{equivalent} if you can get from one to the other by a sequence of chip firing moves.  A divisor is \emph{effective} if it has no vertices with a negative number of chips.  We define $\vert D \vert$ to be the set of effective divisors equivalent to $D$.  Given a vertex $v$, a divisor is \emph{effective away from} $v$ if no vertex other than possibly $v$ has a negative number of chips.

The \emph{degree} of a divisor is the total number of chips.  A divisor $D$ has \emph{positive rank} if, for every vertex $v$, there exists $D' \in \vert D \vert$ such that $v \in \mathrm{Supp}(D')$.  We now come to the main definition of this section.

\begin{definition} 
The \emph{gonality} of a graph is the minimum degree of a divisor with positive rank.
\end{definition}

\subsection{Dhar's Burning Algorithm}

Given a vertex $v$ and a divisor $D$ that is effective away from $v$, there is an algorithm for deciding whether it is equivalent to an effective divisor.  First, we must have the definition of a $v$-reduced divisor.

\begin{definition} 
A divisor $D$ is $v$-\emph{reduced} if it is effective away from $v$ and firing any vertex subset $A$ not containing $v$ results in a divisor that is not effective away from $v$.
\end{definition}

Given a vertex $v$, every divisor is equivalent to a unique $v$-reduced divisor \cite[Corollary~4.13]{BS13}.  A $v$-reduced divisor is equivalent to an effective divisor if and only if it is itself effective.  Thus, it is useful to have an algorithm for computing $v$-reduced divisors.  This algorithm is known as Dhar's burning algorithm (see \cite[Section~5.1]{BS13}).  To run the algorithm, first start a ``fire'' at the vertex $v$.  (This ``fire'' should not be confused with firing a vertex or firing a set of vertices in chip firing.)  Then, every edge adjacent to a vertex that is on fire will burn as well.  If any vertex has fewer chips than it does adjacent edges on fire, that vertex will also burn.  Continue this process until no more vertices can catch fire.  If every vertex of the graph burns, then the divisor is $v$-reduced.  If not, then the set of unburnt vertices can be fired, and the result is a divisor that is effective away from $v$.

The following lemma will be useful for running Dhar's Burning Algorithm on Ferrers rook graphs, because each row and column of such a graph is a complete graph.

\begin{lemma}
\label{Lem:Complete}
Let $D$ be an effective divisor on the complete graph $K_n$.  The following are equivalent:
\begin{enumerate}
\item  \label{Item:Rank} $\rk(D) = 0$,
\item  \label{Item:Chips} for all $i$ in the range $1 \leq i \leq n-1$, $\vert \{ v \mid D(v) \geq i \} \vert < n-i$, and
\item \label{Item:Dhars}  for any vertex $v$ with $D(v) = 0$, running Dhar's Burning Algorithm starting at $v$ burns the entire graph.
\end{enumerate}
\end{lemma}

\begin{proof}
It is clear that \eqref{Item:Dhars} implies \eqref{Item:Rank}.  To see that \eqref{Item:Rank} implies \eqref{Item:Chips}, let $X_i = \{ v \mid D(v) \geq i \}$ and suppose that $\vert X_i \vert \geq n-i$ for some $i$.  For every vertex $v \in X_i$, we have $D(v) \geq i \geq 1$.  Now, let $D'$ be the divisor obtained from $D$ by firing $X_i$.  For every vertex $v \notin X_i$, we have $D'(v) \geq n-i \geq 1$.  Thus $D$ has positive rank.

Finally, to see that \eqref{Item:Chips} implies \eqref{Item:Dhars}, we prove by induction that every vertex in $X_i \smallsetminus X_{i+1}$ burns for all $i$.  The set $X_0 \smallsetminus X_1$, which consists of all vertices that have no chips, clearly burns.  Now, suppose that $ \cup_{j=0}^{i-1} (X_j \smallsetminus X_{j+1})$ burns.  By assumption, this set has size greater than $i$.  The set $X_i \smallsetminus X_{i+1}$ consists of all vertices with exactly $i$ chips.  Since each such vertex is adjacent to each other vertex, and since there are greater than $i$ burnt vertices, each vertex with exactly $i$ chips burns, and the result follows.
\end{proof}

\subsection{Invariants of Divisors on Ferrers Rook Graphs}

To simplify our discussion, we define a few terms related to divisors on Ferrers rook graphs. 

\begin{definition}
Let $F$ be a Ferrers diagram and let $D$ be an effective divisor on $R(F)$.  The \emph{deficit} of a column (resp. row) is the number of vertices in that column (resp. row) minus the number of chips of $D$ in that column (resp. row).

The \emph{poorest} column is the column with the greatest deficit.  Note that multiple columns may be tied for the greatest deficit, in which case they are all poorest columns.

We say that a row or column is \emph{impoverished} if its deficit is greater than or equal to 2.
\end{definition}

The importance of impoverished rows and columns is highlighted by the following important observation. By Lemma~\ref{Lem:Complete}, if we run Dhar's burning algorithm and any vertex in an impoverished row or column burns, then the entire row or column burns.

\section{Bounds on the Gonality of Ferrers Rook Graphs}
\label{Sec:Bounds}

In this section, we compute both upper and lower bounds for the gonality of Ferrers rook graphs.

\subsection{Upper Bounds}
Since the gonality is the minimum degree of a divisor with positive rank, to compute an upper bound, it suffices to find a divisor of positive rank.

Let $F$ be a Ferrers diagram, and let $(x,y) \in \mathbb{N}^2$.  We write $D_{x,y}$ for the divisor on $R(F)$ given by
\[
D_{x,y} = \sum_{\substack{x' \neq x \\ y' \neq y}} (x',y') .
\]
In other words, $D_{x,y}$ is the sum of all vertices in $R(F)$ that are not in column $x$ or row $y$.  We prove a few results about these divisors.

\begin{proposition}
\label{Prop:xyMoreThan1}
Let $F$ be a Ferrers diagram.  If $x, y >1$ and $(x,y) \notin F$, then $D_{x,y}$ has positive rank.
\end{proposition}

\begin{proof}
Let $(x',y') \in F$.  We will show that there is an effective divisor equivalent to $D_{x,y}$ that contains $(x',y')$ in its support.  If $x' \neq x$ and $y' \neq y$, then $(x',y')$ is in the support of $D_{x,y}$.  Otherwise, without loss of generality assume that $x'=x$.  Now fire the complement of column $x$ to obtain an equivalent divisor $D \sim D_{x,y}$.  To see that $D$ is effective, note that because $(x,y) \notin F$, every horizontal edge with one endpoint in column $x$ has a chip at the other endpoint.  Finally, since $x>1$, we see that firing the complement of column $x$ moves a chip from $(x-1,y')$ to $(x,y')$.  Thus, $D$ contains $(x',y')$ in its support.
\end{proof}

\begin{proposition}
\label{Prop:xIs1}
Let $F$ be a Ferrers diagram.  Suppose that the first row of $F$ has length $k$ and the second row of $F$ has length $\ell$.  Then $D_{k+1,1}$ has positive rank if and only if $(\ell,k-\ell+1) \in F$.
\end{proposition}

\begin{proof}
Let $(x,y) \in F$.  We will determine when there is an effective divisor equivalent to $D_{k+1,1}$ that contains $(x,y)$ in its support.  If $y \neq 1$, then $(x,y)$ is in the support of $D_{k+1,1}$.  If $y=1$, then fire the complement of the first row to obtain an equivalent divisor $D \sim D_{k+1,1}$. To see that $D$ is effective, note that every vertical edge with one endpoint in the first row has a chip at its other endpoint.  If $x \leq \ell$, then firing the complement of the first row moves a chip from $(x,2)$ to $(x,1)$, so $D$ contains $(x,y)$ in its support.

Now, suppose that $y=1$ and $x>\ell$.  We show that $D$ is equivalent to an effective divisor containing $(x,y)$ in its support if and only if $(\ell,k-\ell+1) \in F$.  In $D$, the only vertices with chips on them are in the first row, and the number of chips on a vertex is one less than the number of vertices in that column.  If $(\ell,k-\ell+1) \in F$, then each vertex in the first row and the first $\ell$ columns has at least $k-\ell$ chips.  Thus, if we fire the first $\ell$ columns, we obtain an effective divisor with a chip at $(x,y)$.

Conversely, if $(\ell,k-\ell+1) \notin F$, then $D$ has at most $k-\ell-1$ chips at $(\ell,1)$.  Running Dhar's burning algorithm starting at $(x,y)$, we see that vertex $(\ell,1)$ burns, hence so does $(\ell,2)$, and from there, every vertex in the complement of the first row burns.  Each as-yet unbunrt vertex in the top row is adjacent to at least one burning vertex in the top row, and to every vertex in its column.  Since the number of chips on a vertex in the top row is one less than the number of vertices in its column, we then see that the entire graph burns.  It follows that $D$ is not equivalent to an effective divisor with $(x,y)$ in its support.
\end{proof}

We conjecture that divisors of the form $D_{x,y}$ have minimal degree among divisors of positive rank on Ferrers rook graphs. 

\begin{conjecture}
\label{Conj:Gonality}
Let $F$ be a Ferrers diagram.  Suppose that the first row of $F$ has length $k$ and the second row of $F$ has length $\ell$, the first column of $F$ has length $k'$ and the second column has length $\ell'$.  Then
\[
\gon (R(F)) = \min \{ \deg(D_{x,y}) \},
\]
where:
\begin{enumerate}
\item if $(\ell,k-\ell+1), (k'-\ell'+1,\ell') \in F$, then the min is over all $(x,y) \notin F$,
\item if $(\ell,k-\ell+1) \in F$, $(k'-\ell'+1,\ell') \notin F$, then the min is over all $(x,y) \notin F$ with $x>1$,
\item if $(\ell,k-\ell+1) \notin F$, $(k'-\ell'+1,\ell') \in F$, then the min is over all $(x,y) \notin F$ with $y>1$, and
\item if $(\ell,k-\ell+1), (k'-\ell'+1,\ell') \notin F$, then the min is over all $(x,y) \notin F$ with $x,y>1$.
\end{enumerate}
\end{conjecture}

\subsection{Lower Bounds}

Our main tool for computing lower bounds is the following construction.  Let $F$ be a Ferrers diagram and $(x,y) \in F$.  We define
\begin{align*}
L(x,y) &= \{ (x',y') \in F \mid x' \leq x \} \\
U(x,y) &= \{ (x',y') \in F \mid y' \leq y \} .
\end{align*}
The letters $L$ and $U$ stand for ``left'' and ``up'', respectively.  We refer to the first $x$ columns as \emph{$L$-columns} and the first $y$ rows as \emph{$U$-rows}.  By construction, every $L$-column has nontrivial intersection with every $U$-row.  When we draw pictures of these sets, the set $L(x,y)$ will always be depicted with a solid green line and the set $U(x,y)$ with a dashed red line.  See Figure~\ref{Fig:RedGreenRegion}.  We now prove the main result that will be useful for computing lower bounds on the gonality.

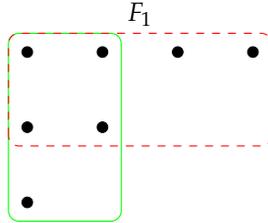
\begin{figure}[ht]
    \vspace{0.25cm}
    \centering
    \begin{tikzpicture}
    \node at (1.5,3.5) {$F_1$};
    \filldraw (0,3) circle (2pt);
    \filldraw (1,3) circle (2pt);
    \filldraw (2,3) circle (2pt);
    \filldraw (3,3) circle (2pt);
    \filldraw (0,2) circle (2pt);
    \filldraw (1,2) circle (2pt);
    %\filldraw (2,2) circle (2pt);
    \filldraw (0,1) circle (2pt);
    %\filldraw (1,1) circle (2pt);
    %\filldraw (0,0) circle (2pt);
    \draw[color=green, rounded corners] (-0.25, 0.75) rectangle (1.25, 3.25) {};
    \draw[color=red, dashed, rounded corners] (-0.25, 1.75) rectangle (3.25, 3.25) {};

    \end{tikzpicture}
    \caption{The regions $L(2,2)$, depicted with a solid green line, and $U(2,2)$, depicted with a dashed red line.}
    \label{Fig:RedGreenRegion}
\end{figure}

\begin{theorem}
\label{Thm:RedGreenRegion}
Let $F$ be a Ferrers diagram, let $(x,y) \in F$, and let $D$ be an effective divisor on $R(F)$.  If 
\[
\deg(D) < \min \{ \vert U(x,y) \vert - y, \vert L(x,y) \vert - x \},
\]
then $\rk(D) = 0$.
\end{theorem}

\begin{proof}
Without loss of generality, we may assume that $D$ has the minimal deficit in the poorest $L$-column among all divisors in $\vert D \vert$.  Among divisors minimizing this deficit, we may further assume that $D$ has the maximum number of chips in the top row.

By assumption on $\deg(D)$, there exists an impoverished $U$-row and impoverished $L$-column.  Run Dhar's Burning Algorithm starting with a vertex $v$ in an impoverished $U$-row.  Let $D'$ be the divisor obtained by firing all of the unburnt vertices.  By Lemma~\ref{Lem:Complete}, the row containing $v$ burns entirely.  Since every $U$-row has nontrivial intersection with every $L$-column, by Lemma~\ref{Lem:Complete}, every impoverished $L$-column will also burn entirely.

First, consider an $L$-column in which not all of the vertices burn.  Suppose this column has $n$ vertices and $k$ of them are unburnt.  By the above, this column is not impoverished.  Because an impoverished $U$-row burns entirely, and every $U$-row has nontrivial intersection with every $L$-column, some vertex in this column burns.  It follows that $1 \leq k \leq n-1$.  Thus, $D'$ has at least $k(n-k) \geq n-1$ chips in this column, so this column remains non-impoverished.  It follows that the only columns that may be impoverished for $D'$ are those that burn completely.

Next, suppose that some vertex in $U(x,y)$ does not burn.  Because every vertex in $U(x,y)$ is adjacent to a vertex in every $L$-column, we see that each $L$-column that burns completely must have at least one more chip in $D'$ than in $D$.  By the previous paragraph, it follows that the poorest $L$-column gains at least one chip.  But this contradicts our choice of $D$, which minimizes the deficit in the poorest $L$-column. 

Finally, suppose that all of the vertices in $U(x,y)$ burn.  If some vertices outside $U(x,y)$ do not burn, then $D'$ has more chips in the top row than $D$ does.  Moreover, each column that burns completely does not lose chips, so for each $L$-column that burns completely, $D'$ has a deficit less than or equal to that of $D$.  This contradicts our choice of $D$, which minimizes the deficit in the poorest $L$-column and maximizes the number of chips in the top row.  On the other hand, if all the vertices burn, then the divisor $D$ is $v$-reduced, despite having no chips at $v$, and thus $D$ has rank zero.
\end{proof}

\section{Examples}

In this section, we verify Conjecture~\ref{Conj:Gonality} for certain families of Ferrers diagrams.  To start, let $m,n \geq 2, \ell < m$.  Our first example will be the Ferrers diagrams:
\[
S(m,n,\ell) = \{ (x,y) \in \mathbb{N}^2 \mid x \leq m, y \leq n-1 \} \cup \{ (x,n) \mid x \leq \ell \} .
\]
The Ferrers diagram $S(m,n,\ell)$ looks like an $m \times (n-1)$ rectangle with an extra partial row.  See Figure~\ref{Fig:RectanglePlus}.  We now compute the gonality of the corresponding Ferrers rook graphs.

\begin{figure}[ht]
    \vspace{0.25cm}
    \centering
    \begin{tikzpicture}
    \filldraw (0,3) circle (2pt);
    \filldraw (1,3) circle (2pt);
    \filldraw (2,3) circle (2pt);
    \filldraw (0,2) circle (2pt);
    \filldraw (1,2) circle (2pt);
    \filldraw (2,2) circle (2pt);
    \filldraw (0,1) circle (2pt);
    \filldraw (1,1) circle (2pt);
    \filldraw (2,1) circle (2pt);
    \filldraw (0,0) circle (2pt);
    \draw[color=red, dashed, rounded corners] (-0.2, 0.75) rectangle (2.2, 3.2) {};
    \draw[color=green, rounded corners] (-0.25, -0.25) rectangle (2.25, 3.25) {};
    \end{tikzpicture}
    \caption{The Ferrers diagram $S(3,4,1)$ with the regions $L(3,3)$ and $U(3,3)$.}
    \label{Fig:RectanglePlus}
\end{figure}
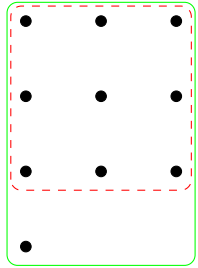

\begin{proposition}
\label{Prop:RectanglePlus}
If $n \geq 3$, we have $\gon(R(S(m,n,\ell))) = \min \{ (m-1)(n-1), mn - 2m + \ell \}$.  If $n=2$, we have $\gon(R(S(m,2,\ell))) = m-1$.
 \end{proposition}

\begin{proof}
By Proposition~\ref{Prop:xyMoreThan1}, $D_{m,n}$ has positive rank.  Since $\deg(D_{m,n}) = (m-1)(n-1)$, we see that
\[
\gon(R(S(m,n\ell))) \leq (m-1)(n-1).
\]
Similarly, if $n \geq 3$, by Proposition~\ref{Prop:xIs1}, $D_{m+1,1}$ has positive rank.  Since $\deg(D_{m+1,1}) = mn-2m+\ell$, we see that 
\[
\gon(R(S(m,n\ell))) \leq mn-2m+\ell
\]
when $r \geq 3$.

For the reverse inequality, note that $\vert L(m,n-1) \vert = mn-m+\ell$ and $\vert U(\ell,n-1) \vert = m(n-1)$.  (See Figure~\ref{Fig:RectanglePlus}.)  By Theorem~\ref{Thm:RedGreenRegion}, it follows that if $\deg(D) < \min \{ mn-2m+\ell, (m-1)(n-1) \}$, then $D$ does not have positive rank.  Hence, $\gon(R(S(m,n\ell))) \geq \min \{ (m-1)(n-1), mn - 2m + \ell \}$.

Finally, if $n=2$, let $D$ be a divisor of degree less than $m-1$.  Without loss of generality, we may assume that $D$ has the maximum number of chips in the top row among all divisors in $\vert D \vert$.  Let $v$ be a vertex in the top row with $D(v) = 0$, and run Dhar's Burning Algorithm starting at $v$.  By Lemma~\ref{Lem:Complete}, the entire top row burns.  If there are any unburnt vertices in the second row, then firing them will increase the number of chips in the top row, contradicting our choice of $D$.  It follows that every vertex burns, hence $D$ is $v$-reduced.  Since $D(v) = 0$, we see that $D$ does not have positive rank.
\end{proof}

The smallest Ferrers diagram that is not of the form $S(m,n,\ell)$ is pictured in Figure~\ref{Fig:L}.  This diagram belongs to the family of ``L-shapes'':
\[
\mathcal{L}(m,n) = \{ (x,1) \mid x \leq m \} \cup \{ (1,y) \mid y \leq n \}.
\]
The Ferrers rook graph $R(\mathcal{L}(m,n))$ is the wedge of two complete graphs.  We compute its gonality here.

\begin{figure}[ht]
    \vspace{0.25cm}
    \centering
    \begin{tikzpicture}
    \filldraw (0,3) circle (2pt);
    \filldraw (1,3) circle (2pt);
    \filldraw (2,3) circle (2pt);
    \filldraw (0,2) circle (2pt);
    \filldraw (0,1) circle (2pt);
        \draw[color=red, dashed, rounded corners] (-0.2, 0.75) rectangle (2.2, 3.2) {};
    \draw[color=green, rounded corners] (-0.25, 0.8) rectangle (0.25, 3.25) {};
    \end{tikzpicture}
    \caption{The Ferrers Diagram $\mathcal{L}(3,3)$ with the regions $L(1,3)$ and $U(1,3)$.}
    \label{Fig:L}
\end{figure}
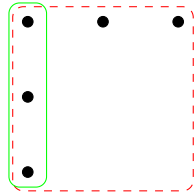

\begin{lemma}
\label{Lem:L}
We have $\gon(R(\mathcal{L}(m,n))) = \max \{m-1, n-1 \}$.
\end{lemma}

\begin{proof}
Without loss of generality, assume that $m \leq n$.  By Proposition~\ref{Prop:xIs1}, we see that $D_{m+1,1}$ has positive rank.  Since $\deg (D_{m+1,n+1}) = n-1$, we see that $\gon(R(L(m,n))) \leq n-1$.  For the reverse inequality, note that $\vert L(1,n) \vert = n$ and $\vert U(1,n) \vert = m+n-1$.  By Theorem~\ref{Thm:RedGreenRegion}, it follows that $\gon(R(L(m,n))) \geq n-1$.
\end{proof}

The smallest Ferrers diagram that is neither of the form $S(m,n,\ell)$ nor of the form $L(m,n)$ is the $3 \times 3$ right triangle pictured in Figure~\ref{fig:TK3}.  We write $T_n$ for the $n \times n$ right triangle:
\[
T_n = \{ (x,y) \in \mathbb{N}^2 \mid x+y \leq n+1 \} .
\]
In general, we do not know how to compute the gonality of $R(T_n)$ for all $n$, though Conjecture~\ref{Conj:Gonality} predicts that $\gon(R(T_n)) = {{n}\choose{2}}$.  We show that this holds for a few small values of $n$.

\begin{figure}[ht]
    \vspace{0.25cm}
    \centering
    \begin{tikzpicture}
    \filldraw (0,3) circle (2pt);
    \filldraw (1,3) circle (2pt);
    \filldraw (2,3) circle (2pt);
    \filldraw (0,2) circle (2pt);
    \filldraw (1,2) circle (2pt);
    \filldraw (0,1) circle (2pt);
    \end{tikzpicture}
    \caption{The Ferrers diagram $T_3$.}
    \label{fig:TK3}
\end{figure}
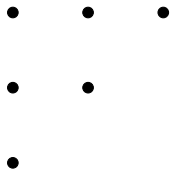

\begin{theorem}
\label{Thm:T3}
We have $\gon(R(T_3)) = 3$.
\end{theorem}

\begin{proof}
By Proposition~\ref{Prop:xyMoreThan1}, we see that $D_{2,3}$ has positive rank.  Since $\deg(D_{2,3}) = 3$, it follows that 
\[
\gon(R(T_3)) \leq 3.
\]
For the reverse inequality, note that $\vert L(2,2) \vert = \vert U(2,2) \vert = 5$.  By Theorem~\ref{Thm:RedGreenRegion}, it follows that 
\[
\gon(R(T_3)) \geq 3.
\]
\end{proof}

\begin{figure}[ht]
    \vspace{0.25cm}
    \centering
    \begin{tikzpicture}
    \filldraw (0,3) circle (2pt);
    \filldraw (1,3) circle (2pt);
    \filldraw (2,3) circle (2pt);
    \filldraw (3,3) circle (2pt);
    \filldraw (0,2) circle (2pt);
    \filldraw (1,2) circle (2pt);
    \filldraw (2,2) circle (2pt);
    \filldraw (0,1) circle (2pt);
    \filldraw (1,1) circle (2pt);
    \filldraw (0,0) circle (2pt);
    %\node at (0,3.5) {a};
    %\node at (1,3.5) {b};
    %\node at (2,3.5) {c};
    %\node at (3,3.5) {d};
    %\node at (-.5,3) {1};
    %\node at (-.5,2) {2};
    %\node at (-.5,1) {3};
    %\node at (-.5,0) {4};
    %\draw[color=green, rounded corners] (-0.25, 0.75) rectangle (2.25, 3.25) {};
    \draw[color=green, rounded corners] (-0.25, -0.25) rectangle (1.25, 3.25) {};
    \draw[color=red, dashed, rounded corners] (-0.25, 0.75) rectangle (3.25, 3.25) {};
    \end{tikzpicture}
    \caption{The Ferrers diagram $T_4$, with the regions $L(2,3)$ and $U(2,3)$.}
    \label{fig:TK4}
\end{figure}
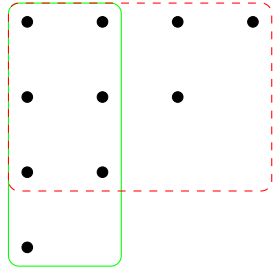

\begin{theorem}
\label{Thm:T4}
We have $\gon(R(T_4)) = 6$.
\end{theorem}

\begin{proof}
By Proposition~\ref{Prop:xyMoreThan1}, we have $\gon(R(T_4)) \leq 6$. Now, let $D$ be a divisor of degree 5.  Note that $\vert L(2,3) \vert = 7$ and $\vert U(2,3) \vert = 9$.  As in the proof of Theorem~\ref{Thm:RedGreenRegion}, we may assume that $D$ has the minimal deficit in the poorest $L$-column among all divisors in $\vert D \vert$.  Among divisors minimizing this deficit, we may further assume that $D$ has the maximum number of chips in the top row.  If one of the $L$-columns is impoverished, then the proof of Theorem~\ref{Thm:RedGreenRegion} shows that $D$ does not have positive rank.  We may therefore assume that neither $L$-column is impoverished.

Because neither $L$-column is impoverished, the first column has exactly 3 chips and the second column has exactly 2 chips.  Now, run Dhar's burning algorithm starting at $v=(4,1)$.  We will show, by case analysis, that every vertex in the first two rows burns.  As a consequence, either the entire graph burns, hence $D$ is not $v$-reduced and therefore does not have positive rank, or some vertex in the bottom two rows is unburnt.  In the latter case, after firing the unburnt vertices, the top row will gain chips and the left two columns will not lose chips, contradicting our assumptions on $D$.

Because they have no chips on them, the vertices $(3,1)$ and $(3,2)$ will burn.  The vertex $(2,1)$ will burn unless it has both of the second column's 2 chips, in which case $(2,2)$ would burn.  If $(2,2)$ burns, then $(2,1)$ is adjacent to 3 burnt vertices, so it must burn as well.  If $(2,1)$ has fewer than 2 chips, then it burns, and $(2,2)$ will burn unless it has 2 chips, in which case $(2,3)$ will burn, and $(2,2)$ will burn thereafter.  By a similar argument, $(1,1)$ and $(1,2)$ must burn. Thus, every vertex in the top two rows burns.
\end{proof}

%\begin{figure}[ht]
%    \vspace{0.25cm}
%    \centering
%    \begin{tikzpicture}
%    \filldraw (0,3) circle (2pt);
%    \filldraw (1,3) circle (2pt);
%    \filldraw (2,3) circle (2pt);
%    \filldraw (3,3) circle (2pt);
%    \filldraw (0,2) circle (2pt);
%    \filldraw (1,2) circle (2pt);
%    \filldraw (2,2) circle (2pt);
%    \filldraw (0,1) circle (2pt);
%    \filldraw (1,1) circle (2pt);
%    \filldraw (0,0) circle (2pt);
%    \filldraw (4,3) circle (2pt);
%    \filldraw (3,2) circle (2pt);
%    \filldraw (2,1) circle (2pt);
%    \filldraw (1,0) circle (2pt);
%    \filldraw (0,-1) circle (2pt);
%    \node at (0,3.5) {a};
%    \node at (1,3.5) {b};
%    \node at (2,3.5) {c};
%    \node at (3,3.5) {d};
%    \node at (4,3.5) {e};
%    \node at (-.5,3) {1};
%    \node at (-.5,2) {2};
%    \node at (-.5,1) {3};
%    \node at (-.5,0) {4};
%    \node at (-.5,-1) {5};
%    \end{tikzpicture}
%    \caption{5x5 triangle rook graph}
%    \label{fig:TK5}
%\end{figure}

\begin{theorem}
\label{Thm:T5}
We have $\gon(R(T_5)) = 10$.
\end{theorem}

\begin{proof}
By Proposition~\ref{Prop:xyMoreThan1}, we have $\gon(R(T_4)) \leq 10$. Now, let $D$ be a divisor of degree 9.  Note that $\vert L(3,3) \vert = \vert U(3,3) \vert = 12$.  As in the proof of Theorem~\ref{Thm:RedGreenRegion}, we may assume that $D$ has the minimal deficit in the poorest $L$-column among all divisors in $\vert D \vert$.  Among divisors minimizing this deficit, we may further assume that $D$ has the maximum number of chips in the top row.  If both one of the $L$-columns and one of the $U$-rows are impoverished, then the proof of Theorem~\ref{Thm:RedGreenRegion} shows that $D$ does not have positive rank.  We may therefore assume that either none of the $L$-columns is impoverished or none of the $U$-rows is impoverished.

First, assume that none of the $L$-columns is impoverished.  It follows that the first column has exactly 4 chips, the second column exactly 3, and the third column exactly 2.  Now, runs Dhar's burning algorithm starting at $v=(5,1)$.  Because they have no chips on them, $(4,1)$ and $(4,2)$ will also burn.  As in the proof of Theorem~\ref{Thm:T4}, every vertex in the top two rows will burn as well.

Since there are only 4 chips in the first column, by Lemma~\ref{Lem:Complete}, the entire column will burn unless all 4 chips are on the same vertex.  Similarly, the entire second column will burn unless all 3 chips are on the same vertex, and the entire third column will burn unless both chips are on the same vertex.  Moreover, these vertices will burn unless every vertex in the same row does not burn.  Thus, the set of unburnt vertices is a union of rows. If we fire every unburnt vertex, then the resulting divisor has more chips in the top row, and every column has the same number of chips, contradicting our assumptions on $D$.

Finally, assume that none of the $U$-rows are impoverished.  Run Dhar's burning algorithm starting at $v'=(1,5)$.  By the same argument as the previous paragraph, we see that the first two columns will burn, and the set of unburnt vertices is a union of columns.  If this set does not contain the third column, then firing the unburnt vertices increases the number of chips in each of the $L$-columns, contradicting our choice of $D$.  If this set does contain the third column, then we must have $D = 4(3,1) + 3(3,2) + 2(3,3)$.  Firing the third column, we see that $D$ is equivalent to $D_{3,4} - v'$.  Running Dhar's burning algorithm starting at $v'$, we see that $D_{3,4} - v'$ is $v'$-reduced, hence $D$ does not have positive rank.
\end{proof}

%\begin{figure}[ht]
 %   \vspace{0.25cm}
  %  \centering
   % \begin{tikzpicture}
 %   \filldraw (0,3) circle (2pt);
  %  \filldraw (1,3) circle (2pt);
   % \filldraw (2,3) circle (2pt);
%    %\filldraw (3,3) circle (2pt);
 %   \filldraw (0,2) circle (2pt);
  %  \filldraw (1,2) circle (2pt);
%    \filldraw (2,2) circle (2pt);
%    \filldraw (0,1) circle (2pt);
%    \filldraw (1,1) circle (2pt);
%    \filldraw (0,0) circle (2pt);
%    \filldraw (4,3) circle (2pt);
%    \filldraw (3,2) circle (2pt);
%    \filldraw (2,1) circle (2pt);
 %   \filldraw (1,0) circle (2pt);
%    \filldraw (0,-1) circle (2pt);
%    \filldraw (0,-2) circle (2pt);
%    \filldraw (1,-1) circle (2pt);
 %   \filldraw (2,0) circle (2pt);
  %  \filldraw (3,1) circle (2pt);
%    \filldraw (4,2) circle (2pt);
 %   \filldraw (5,3) circle (2pt);
%    \node at (0,3.5) {a};
 %   \node at (1,3.5) {b};
%    \node at (2,3.5) {c};
 %   \node at (3,3.5) {d};
  %  \node at (4,3.5) {e};
%    \node at (5,3.5) {f};
 %   \node at (-.5,3) {1};
  %  \node at (-.5,2) {2};
%    \node at (-.5,1) {3};
 %   \node at (-.5,0) {4};
  %  \node at (-.5,-1) {5};
   % \node at (-.5,-2) {6};
  %  \end{tikzpicture}
%    \caption{6x6 triangle rook graph}
 %   \label{fig:TK6}
%\end{figure}

% \begin{theorem}
%$$\gon(TK_6) = 15$.
%\end{theorem}

%\begin{proof}

%\end{proof}

To close, we verify Conjecture~\ref{Conj:Gonality} for all Ferrers diagrams $F$ with $\vert F \vert \leq 8$.

\begin{theorem}
\label{Thm:LessThan8}
Conjecture~\ref{Conj:Gonality} holds for all Ferrers diagrams $F$ with $\vert F \vert \leq 8$.
\end{theorem}

\begin{proof}
Figure~\ref{Fig:LessThan8} depicts all Ferrers diagrams $F$ with $\vert F \vert \leq 8$, up to transpose, that have not been covered by previous cases.

\begin{figure}[ht]
    \vspace{0.25cm}
    \centering
    \begin{tikzpicture}
    \node at (1.5,3.5) {$F_1$};
    \filldraw (0,3) circle (2pt);
    \filldraw (1,3) circle (2pt);
    \filldraw (2,3) circle (2pt);
    \filldraw (3,3) circle (2pt);
    \filldraw (0,2) circle (2pt);
    \filldraw (1,2) circle (2pt);
    %\filldraw (2,2) circle (2pt);
    \filldraw (0,1) circle (2pt);
    %\filldraw (1,1) circle (2pt);
    %\filldraw (0,0) circle (2pt);
    %\draw[color=green, rounded corners] (-0.25, -0.25) rectangle (1.25, 3.25) {};
    %\draw[color=red, dashed, rounded corners] (-0.25, 0.75) rectangle (3.25, 3.25) {};

    \node at (7,3.5) {$F_2$};
    \filldraw (5,3) circle (2pt);
    \filldraw (6,3) circle (2pt);
    \filldraw (7,3) circle (2pt);
    \filldraw (8,3) circle (2pt);
    \filldraw (9,3) circle (2pt);
    \filldraw (5,2) circle (2pt);
    \filldraw (6,2) circle (2pt);
    %\filldraw (7,2) circle (2pt);
    \filldraw (5,1) circle (2pt);
    %\filldraw (6,1) circle (2pt);
    %\filldraw (5,0) circle (2pt);
    %\draw[color=green, rounded corners] (-0.25, -0.25) rectangle (1.25, 3.25) {};
    %\draw[color=red, dashed, rounded corners] (-0.25, 0.75) rectangle (3.25, 3.25) {};

    \node at (1.5,-0.5) {$F_3$};
    \filldraw (0,-1) circle (2pt);
    \filldraw (1,-1) circle (2pt);
    \filldraw (2,-1) circle (2pt);
    \filldraw (3,-1) circle (2pt);
    \filldraw (0,-2) circle (2pt);
    \filldraw (1,-2) circle (2pt);
    \filldraw (2,-2) circle (2pt);
    \filldraw (0,-3) circle (2pt);
    %\filldraw (1,-3) circle (2pt);
    %\filldraw (0,-4) circle (2pt);
    %\draw[color=green, rounded corners] (-0.25, -0.25) rectangle (1.25, 3.25) {};
    %\draw[color=red, dashed, rounded corners] (-0.25, 0.75) rectangle (3.25, 3.25) {};

    \node at (6.5,-0.5) {$F_4$};
    \filldraw (5,-1) circle (2pt);
    \filldraw (6,-1) circle (2pt);
    \filldraw (7,-1) circle (2pt);
    \filldraw (8,-1) circle (2pt);
    \filldraw (5,-2) circle (2pt);
    \filldraw (6,-2) circle (2pt);
    %\filldraw (7,-2) circle (2pt);
    \filldraw (5,-3) circle (2pt);
    \filldraw (6,-3) circle (2pt);
    %\filldraw (5,-4) circle (2pt);
    %\draw[color=green, rounded corners] (-0.25, -0.25) rectangle (1.25, 3.25) {};
    %\draw[color=red, dashed, rounded corners] (-0.25, 0.75) rectangle (3.25, 3.25) {};

    \node at (1.5,-4.5) {$F_5$};
    \filldraw (0,-5) circle (2pt);
    \filldraw (1,-5) circle (2pt);
    \filldraw (2,-5) circle (2pt);
    \filldraw (3,-5) circle (2pt);
    \filldraw (0,-6) circle (2pt);
    \filldraw (1,-6) circle (2pt);
    %\filldraw (2,-6) circle (2pt);
    \filldraw (0,-7) circle (2pt);
    %\filldraw (1,-7) circle (2pt);
    \filldraw (0,-8) circle (2pt);
    %\draw[color=green, rounded corners] (-0.25, -0.25) rectangle (1.25, 3.25) {};
    %\draw[color=red, dashed, rounded corners] (-0.25, 0.75) rectangle (3.25, 3.25) {};
    
    \end{tikzpicture}
    \caption{Ferrers diagrams $F$ with $\vert F \vert \leq 8$, not covered by previous cases.}
    \label{Fig:LessThan8}
\end{figure}
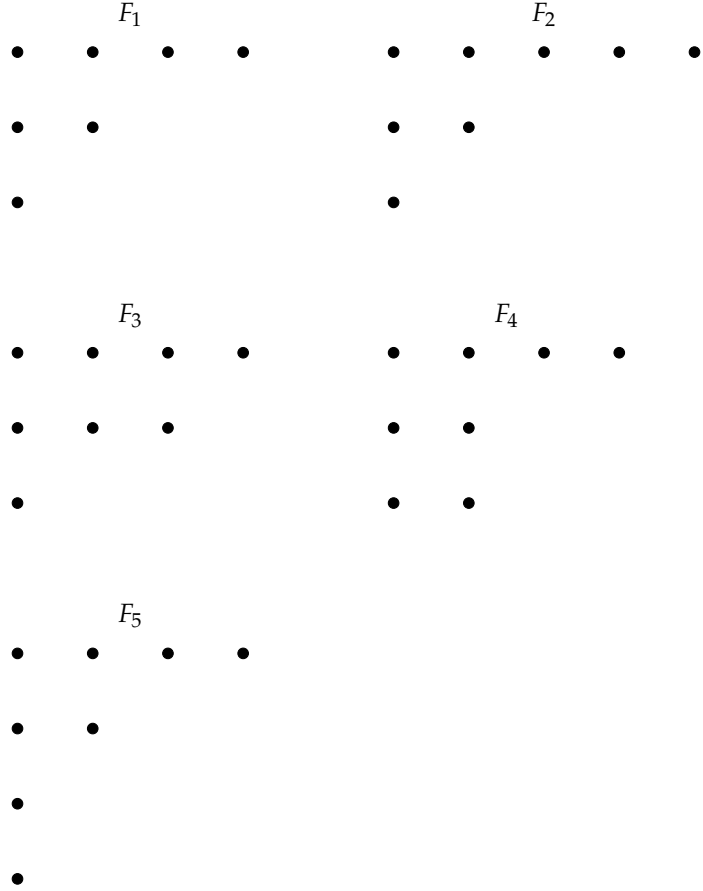

We begin by showing that $\gon(R(F_3)) = \gon(R(F_4)) = 4$.  By Proposition~\ref{Prop:xIs1}, both gonalities are at most 4.  For $F_3$, we see that $\vert L(3,2) \vert = \vert U(3,2) \vert = 7$.  Similarly, for $F_4$, we see that $\vert L(2,3) \vert = 6$ and $\vert U(2,3) \vert = 8$.  In both cases, it follows from Theorem~\ref{Thm:RedGreenRegion} that the gonality is at least 4.

We now show that $\gon(R(F_1)) = 4$.  By Proposition~\ref{Prop:xIs1}, we have $\gon(R(F_1)) \leq 4$.  For the reverse inequality, note that $\vert L(2,2) \vert = 5$ and $\vert U(2,2) \vert = 6$.  Now, let $D$ be a divisor of degree 3.  As in the proof of Theorem~\ref{Thm:RedGreenRegion}, we may assume that $D$ has the minimal deficit in the poorest $L$-column among all divisors in $\vert D \vert$.  Among divisors minimizing this deficit, we may further assume that $D$ has the maximum number of chips in the top row.  If one of the $L$-columns is impoverished, then the proof of Theorem~\ref{Thm:RedGreenRegion} shows that $D$ does not have positive rank.  We may therefore assume that neither $L$-column is impoverished.  This implies that the first column has exactly 2 chips and the second column has exactly one.  Now, run Dhar's burning algorithm starting at $v=(4,1)$.  Then, by Lemma~\ref{Lem:Complete}, the whole top row burns.  Following the proof of Theorem~\ref{Thm:RedGreenRegion}, we see that either the whole graph burns, in which case $D$ does not have positive rank, or $D$ is equivalent to an effective divisor with the same number of chips in each of the first two columns, and more chips in the top row, contradicting our choice of $D$.

The other cases are similar.  By Proposition~\ref{Prop:xIs1}, we have $\gon(R(F_2) \leq 5$.  For the reverse inequality, note that $\vert L(2,2) \vert = 5$ and $\vert U(2,2) \vert = 7$.  Now, let $D$ be a divisor of degree 4.  We again assume that $D$ has the minimal deficit in the poorest $L$-column among all divisors in $\vert D \vert$.  Among divisors minimizing this deficit, we may further assume that $D$ has the maximum number of chips in the top row.  If one of the $L$-columns is impoverished, then the proof of Theorem~\ref{Thm:RedGreenRegion} shows that $D$ does not have positive rank.  We may therefore assume that neither $L$-column is impoverished.  Now, run Dhar's burning algorithm starting at a vertex with no chip in the top row.  There are at least two such vertices, and the other burns as well.  Now there are two cases: either there is a third vertex in the top row that does not have a chip, or the vertex $(2,1)$ has only one chip.  In either case, by Lemma~\ref{Lem:Complete}, the whole top row burns.  The rest of the argument follows exactly as in the previous paragraph.

Finally, we consider $F_5$.  By Proposition~\ref{Prop:xIs1}, we have $\gon(R(F_5) \leq 5$.  For the reverse inequality, note that $\vert L(2,2) \vert = \vert U(2,2) \vert = 6$.  Now, let $D$ be a divisor of degree 4.  We again assume that $D$ has the minimal deficit in the poorest $L$-column among all divisors in $\vert D \vert$.  Among divisors minimizing this deficit, we may further assume that $D$ has the maximum number of chips in the top row.  If one of the $L$-columns and one of the $U$-rows is impoverished, then the proof of Theorem~\ref{Thm:RedGreenRegion} shows that $D$ does not have positive rank.  If neither $L$-column is impoverished, then the first column has exactly 3 chips and the second column has exactly 1, and the rest of the proof is exactly as it was for $F_1$.  Finally, if neither $U$-row is impoverished, then the first row has exactly 3 chips and the second row has exactly 1.  Starting a fire at $(1,4)$ then results in either burning the whole graph, or producing an equivalent divisor with a smaller deficit in the poorest $L$-column, contradicting our choice of $D$.
\end{proof}

\end{document}